\documentclass[preprint,12pt,number,english]{elsarticle}
\usepackage{yhmath}
\usepackage{amsmath,amssymb,amsthm}
\usepackage{xcolor,soul} 
\usepackage{array,dcolumn}
\usepackage{enumitem}
\usepackage{cancel}
\usepackage[normalem]{ulem} 
\usepackage{lineno}
\usepackage{babel}
\usepackage[T1]{fontenc}
\newtheorem{thm}{Theorem}
\newtheorem{propo}{Proposition}
\newtheorem{lema}{Lemma}
\newtheorem{coro}{Corollary}
\newtheorem{defi}{Definition}
\newtheorem{rmk}{Remark}
\newtheorem{ejem}{Example}

\newcounter{letapar}

\begin{document}
\bibliographystyle{plain}
\begin{frontmatter}
\title{Almost strictly sign regular rectangular matrices}



\author[1]{P. Alonso}
\ead{palonso@uniovi.es}
\author[2]{J.M. Pe\~na}
\ead{jmpena@unizar.es}
\author[1]{M.L. Serrano}
\ead{mlserrano@uniovi.es}

\address[1]{Departamento de Matem\'aticas, Universidad de Oviedo, Spain}
\address[2]{Departamento de Matem\'atica Aplicada, Universidad de Zaragoza, Spain}

\begin{abstract}
Almost strictly sign regular matrices are sign regular matrices with a special zero pattern and whose nontrivial minors are nonzero. In this paper we provide several properties of almost strictly sign regular rectangular matrices of maximal rank and analyze their $QR$ factorization.
\end{abstract}

\end{frontmatter}

\section{Introduction}

Given a matrix $A$, the variation diminishing property guarantees that the number of variations of sign in the consecutive components of the image vector $Ax$ is bounded above by the number of variations of sign in the consecutive components of $x$. This property characterizes sign regular matrices among matrices of maximal rank (see Corollary 5.4 of \cite{Ando1987}) and has many important applications in several fields such as Statistics, Approximation Theory, Mechanics, Economy or Computer Aided Geometric Design. An important subclass of sign regular matrices is the class of almost strictly sign regular matrices, which present a special zero pattern and whose nontrivial minors are nonzero. Some properties of square almost strictly sign regular matrices have been studied (see, for instance, {\cite{Adm2016,Alonso2017,Alonso2015,RongHuang2012}}).

In this paper we extend the concept of almost strictly sign regular matrices to the rectangular case of maximal rank and we analyze several properties. A motivation for this study comes from the fact that least squares problems lead to coefficient matrices of {overdetermined} systems. If those coefficient matrices have maximal rank, then the $QR$ factorization is a key tool to solve these problems. In this paper we provide several characterizations of almost strictly sign regular matrices and analyze their $QR$ factorization.

We now present the layout of the paper. Section 2 includes some basic notations and some definitions involving the zero pattern of a matrix. Section 3 introduces the main classes of matrices consider in this paper and presents several characterizations of almost strictly sign regular matrices with maximal rank. Finally, the analysis of the $QR$ factorization of  almost strictly sign regular matrices of maximal rank is performed in Section 4.

\section{Basic {notation}}

First, we will introduce some concepts and notations that will be used throughout the document.

For $k, n \in \mathbb{N}$, with $1 \leq k \leq n$, $Q_{k,n}$ denotes the set of all increasing sequences of $k$ natural numbers not greater than $n$. So, given $m,n,k \in \mathbb{N}$ fulfilling that $k \leq \min \{m , n \}$, the elements $\alpha=(\alpha_{1}, \dots, \alpha_{k})\in Q_{k,m}$ and  $\beta=(\beta_{1}, \dots, \beta_{k})\in Q_{k,n}$, if $A$ is a real $m \times n$ matrix, we denote by $A[\alpha|\beta]$ the $k \times k$ submatrix of $A$ containing rows $\alpha_{1}, \dots, \alpha_{k}$ and columns $\beta_{1}, \dots, \beta_{k}$ of $A$. If $\alpha=\beta$, we denote by $A[\alpha]:=A[\alpha|\alpha]$ the corresponding principal submatrix. Besides, $Q^{0}_{k,n}$ denotes the set of increasing sequences of $k$ consecutive natural numbers less than or equal to $n$.

For a sequence $\alpha \in Q_{k,m}$, we denote
$$ d(\alpha)=\sum_{i=1}^{k-1}(\alpha_{i+1}-\alpha_{i}-1)=\alpha_{k}-\alpha_{1}-(k-1)$$
with the convention $d(\alpha)=0$ if $k=1$. Note that if $d(\alpha)=0$ then $\alpha \in Q^{0}_{k,n}$.

Now, type-I (type-II) staircase matrices and some results related {to} them are presented.

\begin{defi}\label{tipoI}
Let $A=\left(a_{ij}\right)_{1 \leq i \leq m, 1 \leq j \leq n}$ be an $m\times n$ real matrix. $A$ is called type-I staircase if it {simultaneously} satisfies  the following conditions:
\begin{eqnarray} \label{diagtipoI}
\hbox{if} \quad i=j & \Rightarrow & a_{ii}\not=0 , \\ \label{abajotipoI}
\hbox{if} \quad i>j \quad \hbox{and} \quad  a_{ij}=0&\Rightarrow & a_{kl}=0 , \ \hbox{if} \ k\geq  i\hbox{, } l \leq j,\\ \label{arribatipoI}
\hbox{if} \quad i<j  \quad \hbox{and} \quad  a_{ij}=0 &\Rightarrow &  a_{kl}=0 , \hbox{if} \ k\leq  i\hbox{, } l \geq j.
\end{eqnarray}
\end{defi}

Now, {the concept of} backward identity matrix  is introduced.

\begin{defi}\label{def_Pn}
An  $r \times r$ matrix $P_r=(p_{ij})_{1\leq i,j \leq r}$ is called {a} backward identity matrix if the element $(i,j)$ of the matrix $P_r$ is defined in the form
$$
p_{ij}=\left\{\begin{array}{lr}
                               1, & \qquad \mathrm{if } \  i+j=r+1,\\
                               0, &   \mathrm{otherwise},
                               \end{array}\right.
                               $$
i.e.
$$
P_r=\left(
      \begin{array}{ccccc}
        0 & 0 & \cdots & 0 & 1 \\
        0 & 0 & \cdots & 1 & 0 \\
        \vdots &  \vdots & \vdots &  \vdots & \vdots \\
        0 & 1 & \cdots & 0 & 0 \\
        1 & 0 & \cdots & 0 & 0\\
      \end{array}
    \right).
$$
\end{defi}

\begin{rmk}
Note that the left product by the matrix $P_m$ produces a permutation in the rows of the matrix $A$, so, the row $k$ of $P_m A$ is the row $m-k+1$ of matrix $A$, with $k= 1,\dots, m$.
\end{rmk}

\begin{defi}
A matrix $A=\left(a_{ij}\right)_{1 \leq i \leq m, 1 \leq j \leq n}$ is called type-II staircase if it satisfies that $P_m A$ is a type-I staircase matrix.
\end{defi}

\begin{rmk}
Note that this is a generalization of the type-I and type-II staircase matrices of a square matrix \cite{Alonso2015,RongHuang2012}.
\end{rmk}

In order to describe the zero pattern of a nonsingular type-I staircase (or type-II staircase) {matrix $A$}, it is adequate to introduce the {following} sets of indices that were presented for $ n \times n $ matrices by M. Gasca and J. M. Peña in \cite{Char2006}.

\begin{defi}\label{notacion_Iescalonada}
Let $A=\left(a_{ij}\right)_{1 \leq i \leq m, 1 \leq j \leq n}$ be a type-I staircase matrix and  $r =\min \{m,n \}$. We denote
\begin{equation}\label{def_ik}
i_0=1, \qquad j_0=1,
\end{equation}
and for $k=1,\dots ,l$, where $l$ will be the first integer such that $j_l \geq r+1$, we define
\begin{equation}
i_k=\max \left\{i  \ / \ a_{ij_{k-1}}\not= 0 \right\}+1 ,
\end{equation}
if $i_k<r+1$ then
\begin{equation}\label{def_jk}
j_k=\max \left\{j \leq i_k 
\ / \ a_{i_{k}j} = 0 \right\}+1,
\end{equation}
if $i_{k}=r+1$ then $j_k=n+1$.

{By} applying the same process to $A^T$ the indices $\overline{i}_k$ and $\overline{j}_k$ are obtained. Denoting by $\widehat{i}_k=\overline{j}_k$ and $\widehat{j}_k=\overline{i}_k$ we have the index sets
$$
\begin{array}{ccccccc}
  I &=& \left\{i_0,i_1,\dots , i_l\right\}, & \qquad &
  J &=& \left\{j_0,j_1,\dots , j_l\right\}, \\
  \widehat{I} &=& \left\{\widehat{i}_0,\widehat{i}_1,\dots , \widehat{i}_s\right\}, & \qquad &
  \widehat{J} &=& \left\{\widehat{j}_0,\widehat{j}_1,\dots , \widehat{j}_s\right\}.
  \end{array}
$$
that describe the zero-nonzero pattern of the matrix $A$.
\end{defi}

\begin{ejem}\label{ejempatron}
The zero-nonzero pattern of the matrix
$$
A=\left(
    \begin{array}{ccccc}
      1 & 1& 0 & 0 & 0 \\
      1 & 2 & 0 & 0 & 0 \\
      0 & 1 & 1 & 1 & 0 \\
      0 & 0 & 0 & 1 & 1 \\
      0 & 0 & 0 & 2 & 3 \\
      0 & 0 & 0 & 4 & 5 \\
      0 & 0 & 0 & 0 & 0 \\
    \end{array}
  \right)
$$
is given by the index sets $I=\left\{1,3,4,7\right\}$, $J=\left\{1,2,4,6\right\}$. {By} applying the process for the matrix
$$
A^T=\left(
      \begin{array}{ccccccc}
        1 & 1 & 0 & 0 & 0 & 0 & 0 \\
        1 & 2 & 1 & 0 & 0 & 0 & 0 \\
        0 & 0 & 1 & 0 & 0 & 0 & 0 \\
        0 & 0 & 1 & 1 & 2 & 4 & 0 \\
        0 & 0 & 0 & 1 & 3 & 5 & 0 \\
      \end{array}
    \right)
$$
we have $\overline{I}=\{ 1,3,5,6\}$, $\overline{J}=\{1,3,4,8\}$, so that $\widehat{I}=\left\{1,3,4,8\right\}$ and $\widehat{J}=\left\{1,3,5,6\right\}$.
\end{ejem}

\begin{rmk}
Note that in {the case where} $m=n$ then this definition of zero pattern coincides with that defined in \cite{Char2006}.
\end{rmk}

In addition, for subsequent results, it is also necessary to introduce the indices $j_t$ and $\widehat{i}_t$ (see Theorem 2.4 of \cite{Char2006}).

\begin{defi}\label{defi_j_t,igorro_t}
Let $A=\left(a_{ij}\right)_{1 \leq i \leq m, 1 \leq j \leq n}$ be a real $m \times n$ matrix, type-I staircase, with zero pattern $I$, $J$, $\widehat{I}$ and $\widehat{J}$.

If $j \leq i$ we define
\begin{equation}\label{def_jt}\index{$j_t$}
  j_t=\max \left\{j_s \ / \ 0 \leq s \leq k-1, \ j-j_s \leq i-i_s \right\},
\end{equation}
{with $k$ being the only} index satisfying that $j_{k-1} \leq j < j_k$. We will denote {by $i_t$\index{$i_t$} the $t$-th element} of $I$, where $t$ is the index of $j_t$.

If $j > i$  we define
\begin{equation}\label{def_itgorro}\index{$\widehat{i}_t$}
\widehat{i}_t=\max \left\{\widehat{i}_s \ / \ 0 \leq s \leq k'-1\ , i-\widehat{i}_s \leq j-\widehat{j}_s \right\},
\end{equation}
{with $k'$  being the only} index satisfying that $\widehat{i}_{k'-1} \leq i < \widehat{i}_{k'}$. We will denote {by $\widehat{j}_t$ the $t$-th element} of $\widehat{J}$, where $t$ is the index of $\widehat{i}_t$.
\end{defi}

Next, the concept of a nontrivial submatrix is introduced, which is necessary in order to subsequently define the almost strictly sign regular matrices.

\begin{defi}
Let $A=\left(a_{ij}\right)_{1 \leq i \leq m, 1 \leq j \leq n}$ be a real $m \times n$ type-I staircase matrix and $k \leq \min \{m,n\}$. A submatrix $A[\alpha | \beta ]$, with $\alpha \in Q_{k,m}$ and $\beta \in Q_{k,n}$ is nontrivial if all $i \in \{1,2,\dots , k\}$ fulfills that $a_{\alpha_i,\beta_i}\not=0$. Otherwise we will say that $A[\alpha | \beta ]$ is a trivial submatrix.
\end{defi}

\begin{defi}
Let $A=\left(a_{ij}\right)_{1 \leq i \leq m, 1 \leq j \leq n}$ be a real $m \times n$ type-II staircase matrix and $k \leq \min \{m,n\}$. A submatrix $A[\alpha | \beta ]$, with $\alpha \in Q_{k,m}$ and $\beta \in Q_{k,n}$ is nontrivial if $A[\alpha | \beta ]^T$ is a nontrivial submatrix of $A^T$.
\end{defi}

The minor associated {with} a nontrivial submatrix ($A[\alpha | \beta ]$) is called {a} nontrivial minor ($\det A[\alpha | \beta]$).

In \cite{Char2006} the authors define the concept of column boundary minor, characterizing in Theorem 2.4 the almost strictly positive matrices $n \times n$ from the sign of these minors, which {significantly} reduces the number of minors to study. Next, we will proceed to the generalization of this concept when an $m \times n$ real matrix is considered.

\begin{defi}\label{sub-frontera}
Given an $m \times n$ type-I staircase matrix $A=\left(a_{ij}\right)_{1 \leq i \leq m, 1 \leq j \leq n}$, let $B:=A[\alpha |\beta]$ {be}  a submatrix of $A$ with $\alpha \in Q_{k,m}^{0}$, $\beta \in Q_{k,n}^{0}$, $k \in \{1,\dots,$ $\min\{m,n\}\}$ and $a_{\alpha_1,\beta_1}\not=0$, $a_{\alpha_2,\beta_2}\not=0, \dots , a_{\alpha_k,\beta_k}\not=0$. Then $B$ is a column boundary submatrix if either $\beta_1=1$ or $\beta_1>1$ and $A[\alpha|\beta_1 -1 ]= 0$.

Analogously, $B$ is a row boundary submatrix if either $\alpha_1=1$ or $\alpha_1>1$ and $A[\alpha_1 -1 |\beta ]= 0$.
\end{defi}

If $A$ is a type-II staircase matrix, then $B=A[\alpha |\beta]$ is a column boundary submatrix of $A$ if $P_k B$ is a column boundary submatrix of $P_mA$.

Minors corresponding to column or row boundary submatrices are called, respectively,
column or row boundary minors.

The boundary submatrices (minors) that verify that $\alpha_1=1$ or $\beta_1=1$ are called initial boundary submatrices (minors) of $A$.

The zero pattern of a type-I staircase matrix defines the rows and columns involved in the boundary minors of a matrix, as we indicate in the following note.

\begin{rmk}
Let $A$ be an $m \times n$ type-I staircase matrix with zero pattern defined by $I$, $J$, $\widehat{I}$, $\widehat{J}$ and let $B:=A[\alpha |\beta]$ {be} a submatrix of $A$ with $\alpha \in Q_{k,m}^{0}$, $\beta \in Q_{k,n}^{0}$, $k \in \{1,\dots , \min\{m,n\}\}$ and $a_{\alpha_1,\beta_1}\not=0$, $a_{\alpha_2,\beta_2}\not=0, \dots , a_{\alpha_k,\beta_k}\not=0$. Then $B$ is a column boundary submatrix if there exists $k \geq 1$ such that $\beta_1=j_k$ and $\alpha_1\geq i_k$. Analogously, $B$ is a row boundary submatrix if there exists $k\geq 1$ such that $\alpha_1=\widehat{i}_k$ and $\beta_1\geq \widehat{j}_k$.
\end{rmk}

Next, the concept of column generalized boundary submatrix is defined.

\begin{defi}
Given an $m \times n$ matrix $A=\left(a_{ij}\right)_{1 \leq i \leq m, 1 \leq j \leq n}$, let $B:=A[\alpha |\beta]$ a submatrix of $A$ with $\alpha \in Q_{k,m}^{0}$, $\beta \in Q_{k,n}^{0}$, $k \in \{1,\dots,$ $\min\{m,n\}\}$ and $a_{\alpha_1,\beta_1}\not=0$. Then $B$ is a column generalized boundary submatrix if either $\beta_1=1$ or $\beta_1>1$ and $A[\alpha|\beta_1 -1 ]= 0$.

Analogously, $B$ is a row generalized boundary submatrix if either $\alpha_1=1$ or $\alpha_1>1$ and $A[\alpha_1 -1 |\beta ]= 0$.
\end{defi}

Minors corresponding to column (row) generalized boundary submatrices are called column (row) generalized  boundary minors.

\section{Almost strictly sign regular rectangular matrices}

Almost strictly sign regular square matrices have been studied in recent years by different authors (see, for instance \cite{Adm2016,Alonso2015,RongHuang2012}). Next we are going to perform a study of a generalization of this concept to rectangular matrices. For this purpose, some previous concepts necessary for the following results are presented below.

\begin{defi}\label{defsig}
Given a vector $\varepsilon=(\varepsilon_1, \varepsilon_2, \dots , \varepsilon_ n) \in \mathbb{R}^n$, we say that $\varepsilon$ is a signature sequence, or simply, is a signature, if $\varepsilon_i=\pm 1$ for all  $i \leq n$.
\end{defi}

Next, the regular sign matrices (SR), the totally positive matrices (TP) and the almost strictly regular matrices (ASSR) will be defined.

\begin{defi}\label{defSR}
Let $A$ be an $m \times n$ matrix with $\hbox{rank} (A)=r= \min \{m,n\}$.  $A$ is said to be SR with signature $\varepsilon=(\varepsilon_1, \varepsilon_2, \dots , {\varepsilon_ r})$ if it is type-I or type-II staircase and all its minors  satisfy that
\begin{equation}\label{def_sr_gen}
    \varepsilon_k\det A\left[\alpha | \beta \right]\geq 0, \quad \alpha \in Q_{k,m} , \beta \in Q_{k,n}, \quad k \leq r.
\end{equation}
\end{defi}

\begin{defi}\label{TP}
Let $A$ be an $m \times n$ matrix with $\hbox{rank} (A)=r= \min \{m,n\}$. $A$ is said to be TP if {it} is SR with $\varepsilon=(1,1,\dots , 1)$.
\end{defi}

\begin{defi}\label{defASSR}
Let $A$ be an $m \times n$ matrix with $\hbox{rank} (A)=r= \min \{m,n\}$. $A$ is said to be ASSR with signature $\varepsilon=(\varepsilon_1, \varepsilon_2, \dots , {\varepsilon_ r})$ if it is either type-I or type-II staircase and all its nontrivial minors $\det A[\alpha | \beta ]$ satisfy that
\begin{equation}\label{def_assr_gen}
  \varepsilon_k\det A\left[\alpha | \beta \right]>0, \quad \alpha \in Q_{k,m} , \beta \in Q_{k,n}, \quad k \leq r.
\end{equation}
\end{defi}

Starting from the definition of ASSR matrices, it is possible to prove the following results:

\begin{propo}\label{ASSRtipo-IrectSubtb}
Let $A$ be an $m \times n$ matrix with $\hbox{rank} (A)=r= \min \{m,n\}$. If $A$ is an ASSR and type-I staircase matrix, then any nontrivial square submatrix of $A$ is an ASSR matrix.
\end{propo}

\begin{proof}
Let $B=A\left[\alpha |\beta \right]$ be a nontrivial submatrix of $A$. Being a nontrivial submatrix it holds that $a_{\alpha_i,\beta_i}\not=0$ for all $i$. Let us suppose that $a_{\alpha_i,\beta_j} =0$, then, by Definition \ref{tipoI}
\begin{itemize}
  \item If $\alpha_i > \beta_j$, then $a_{kl}=0 \; \forall k \geq \alpha_i$, $ l \leq \beta_j $ and so $a_{\alpha_k,\beta_l}=0$ for each $\beta_l \leq \beta_j,\  \alpha_i \leq \alpha_k$.
  \item If $\alpha_i < \beta_j$, $a_{kl}=0$ $\forall k \leq \alpha_i, \ l \geq \beta_j$ then $a_{\alpha_k,\beta_l}=0$ for each $\beta_l \geq \beta_j,\  \alpha_i \geq \alpha_k$.
\end{itemize}
Thus, $B$ is a type-I staircase submatrix of $A$. Therefore, any nontrivial submatrix of $B$ is a nontrivial submatrix of $A$ and by Definition \ref{defASSR} it satisfies that the nontrivial minors of order $s$ of $B$ are the same sign 
{as} $\varepsilon_s$, and so $A$ is an ASSR matrix.
\end{proof}

\begin{coro}\label{coro1}
Let $A$ be an $m \times n$ matrix with $\hbox{rank} (A)=r= \min \{m,n\}$. If $A$ is an ASSR and type-II staircase matrix, then any nontrivial square submatrix of $A$ is an ASSR matrix.
\end{coro}

\begin{proof}
The proof follows from taking $P_m A$, which is type-I, and applying the above result.
\end{proof}

\begin{propo}\label{signaturasubmatriz}
Let $A$ be an $m \times n$ matrix with $\hbox{rank} (A)=r= \min \{m,n\}$. Then $A$ is ASSR with signature $\varepsilon=(\varepsilon_1,\dots , \varepsilon_r)$ if and only if any nontrivial submatrix of order $k$ of $A$, with $k \leq r$, is an ASSR matrix with signature $\varepsilon=(\varepsilon_1,\dots , \varepsilon_k)$.
\end{propo}

\begin{proof}
By Proposition \ref{ASSRtipo-IrectSubtb} and Corollary \ref{coro1} if $A$ is an ASSR matrix, then its nontrivial square submatrices are also ASSR matrices. Since each nontrivial minor of those submatrices is a nontrivial minor of $A$, the sign must necessarily be that of the corresponding signature of $A$, so the direct implication is proved.

Let us suppose now that it is satisfied that each nontrivial submatrix of order $k$ of $A$, with $k \leq r$, is an ASSR matrix with signature $\varepsilon=(\varepsilon_1,\dots , \varepsilon_k)$. Then, if $A\left[\alpha |\beta\right]$ is a nontrivial submatrix of $A$, applying the hypothesis, $\varepsilon_k\det A\left[\alpha |\beta\right]>0$, therefore it is true that $A$ is ASSR with signature $\varepsilon=(\varepsilon_1,\dots , \varepsilon_r)$, and the theorem has been proved.
\end{proof}

By generalizing Theorem 10 of \cite{RongHuang2012} we can reduce the number of minors that characterize this type of matrices. For this we will need to prove the following lemma:

\begin{lema}\label{lema9}
Let $A$ be an $m \times n$ type-I staircase matrix, $\hbox{rank} (A)=r= \min \{m,n\}$, $\varepsilon=(\varepsilon_1,\varepsilon_2,\dots,\varepsilon_r )$ a signature sequence and  {let us} denote by $h$ the value defined as
$$
  h=\min \left\{ |i-j| \ / \ a_{ij}=0, \ i \in \{1,2,\dots,m\},\ j\in \{1,2,\dots,n\} \right\}.
$$
If all the nontrivial minors of $A$ verify that
 $$
  \varepsilon_k\det A\left[\alpha |\beta \right] >0, \quad \alpha \in Q^0_{k,m}, \quad \beta \in Q^0_{k,n},\quad 1\leq k \leq r,
 $$
then
$$
  \varepsilon_2=\varepsilon_1^2, \ \varepsilon_3 = \varepsilon_1^3,\cdots, \varepsilon_{r-h+1}=\varepsilon_1^{r-h+1}.
$$
\end{lema}

\begin{proof}
Since $A$ is a type-I staircase matrix, then $a_{ii}\not=0$ for all $i\in \{1,2,\dots,$ $\min\{m,n\}\}$, and $h>0$. Without loss of generality, suppose that $a_{ts}=0$, with $t<s$ is such that $h=s-t$. Now, we prove by induction on $k$ that $\varepsilon_k=\varepsilon_1^{k}$.

For $k=1$ the result is trivial. Let us suppose that the result is true for $k$ less than $k-1$ and we will prove that it is true for $k$. Since $A$ is type-I staircase and $a_{ts}=0$ with $t<s$, then $a_{ij}=0$ for $i \leq t,\ s \leq j$ (shadow up from of $a_{ii}$, down it cannot be because $a_{ii}\not= 0$). Then, as $h$ is minimum
 $$
  a_{1h},a_{2,h+1},\dots,a_{k,h+k-1} \not=0,
 $$
in this way, $A\left[1,\dots,k|h,\dots,h+k-1\right]$ is a nontrivial submatrix. Besides, the submatrix $A[1,\ldots,k|h,\ldots,h+k-1]$ has the form
{\small
 $$
\left[\begin{array}{cc}
A\left[1,\dots,t|h,\dots,s-1\right] & 0 \\
* & A\left[t+1,\dots,h+k-1|s,\dots ,h+k-1\right]
\end{array}\right],
$$
}
then
$$
  \hspace*{-6cm}\det A\left[1,\dots,k|h,\dots ,h+k-1\right]=
$$
$$
\hspace{1cm}=\det A\left[1,\dots,t|h,\dots,s-1\right]\det A\left[t+1,\dots,h+k-1|s,\dots ,h+k-1\right].
$$
The minors on the right side of the  equality are non-trivial and of order less than $k$, and applying induction hypothesis, we obtain that
$$
  \varepsilon_{k}=\varepsilon_t\varepsilon_{h+k-1-s+1}\underbrace{=}_{h=s-t} \varepsilon_t\varepsilon_{k-t}=\varepsilon_1^t\varepsilon_1^{k-t}=\varepsilon_1^{k},
$$
and the result is proved.
\end{proof}

\begin{thm}\label{TeoremaMenores}
Let $A$ be an $m \times n$ ($m\geq n$), $\hbox{rank} (A)=r= \min \{m,n\}$ and $\varepsilon=(\varepsilon_1, \dots ,  {\varepsilon_r})$ a signature sequence. Then, $A$ is an ASSR matrix with signature $\varepsilon$ if and only if $A$ is type-I or type-II staircase, and satisfies
\begin{equation}\label{caractASSRMenores}
  \varepsilon_k\det A\left[\alpha | \beta \right]>0, \quad \alpha \in Q^0_{k,m} ,\quad \beta \in Q^0_{k,n}, \quad k \leq n.
\end{equation}
\end{thm}

\begin{proof}
Let us suppose that $A$ is an ASSR matrix with signature $\varepsilon=(\varepsilon_1, \dots , \varepsilon_r)$. Since $Q_{k,n}^0 \subseteq Q_{k,n}$ and $Q_{k,m}^0 \subseteq Q_{k,m}$, it is evident that, by Definition \ref{defASSR}, (\ref{caractASSRMenores}) also holds.

For the reciprocal, suppose that $A$ is type-I or type-II staircase and that (\ref{caractASSRMenores}) holds and we will prove that $A$ is an ASSR matrix with signature $\varepsilon$. Let $A\left[\alpha|\beta\right]$ be a nontrivial submatrix of $A$ with $\alpha \in Q_{k,m}$ and $\beta \in Q_{k,n}$.

Now, we will prove the result by induction on $k$, the order of  {the} minor.

If $k=1$, since $Q_{1,m}=Q_{1,m}^0$, the result is trivial.

Suppose the result is true for $k-1$, i.e., $\varepsilon_{k-1}\det A\left[\alpha|\beta\right]>0$ for  {any} minor such that $\alpha\in Q_{k-1,m}$, $\beta \in Q_{k-1,n}$, and we will prove that the property is true for $k$. First, we suppose that $\alpha \in Q_{k,m}^0$. Let $\beta \in Q_{k,n}$ be and $\ell=d(\beta)(=\beta_k-\beta_1-(k-1))$.

Next, we will prove (\ref{caractASSRMenores}) by induction on $\ell$. If $\ell =0$ then $\beta \in Q^0_{k,m}$ and the result is fulfilled.

Suppose that $A\left[\alpha | \beta \right]$ is a nontrivial submatrix and $d(\theta ) \leq \ell -1$, with $\alpha \in Q^0_{k,m}$ and $\theta \in Q_{k,n}$. In this case, the result is satisfied.

Let $\beta=(\beta_1,\beta_2,\dots,\beta_k)\in Q_{k,n}$ be with $d(\beta )=\ell \geq 1$, such that $A\left[\alpha|\beta\right]$ is a nontrivial submatrix. Let us take $\tau=(\beta_2,\dots,\beta_{k-1})$, as $d(\beta)=\ell \geq 1$, then, there exists $p \notin \tau$ such that $\beta_1<p<\beta_k$.

Using the equality 1.39 of \cite{Ando1987} we have
 \begin{equation}\label{igAndo}
  \det A\left[w|\tau \cup \{p\}\right] \det A\left[\alpha|\tau \cup \{\beta_1,\beta_k\}\right]=
  \end{equation}
   \begin{equation*}
=\det A\left[w|\tau \cup \{\beta_k\}\right]\det A\left[\alpha |\tau \cup \{\beta_1,p\}\right]+\det  A\left[w|\tau \cup \{\beta_1\}\right]\det A\left[\alpha |\tau \cup \{p,\beta_k\}\right]
  \end{equation*}
and this is true for all $w\in Q_{k-1,m}$, with $w\subset \alpha$.

If $A$ is type-I staircase, we are going to distinguish two cases:
\begin{description}
  \item[Case (1)] Suppose that $a_{\alpha_t,\beta_{t+1}}\not=0$, $\forall t =1,2,\dots,k-1$.

      Since $d(\beta)=\ell \geq 1$ then we can  {guarantee} that there  {exists} $1 \leq i\leq k-1$ such that $\beta_i <p<\beta_{i+1}$, and therefore in this case
      $$
      d(\tau \cup \{\beta_1,p\}) \leq \ell-1 , \quad d(\tau \cup \{p,\beta_k\}) \leq \ell-1
      $$
      by the induction hypothesis we can write
      $$
      \varepsilon_k\det A\left[\alpha | \tau \cup \{\beta_1,p\}\right] \geq 0.
      $$

     Given that $A\left[\alpha |\beta\right]$ is a nontrivial submatrix of $A$, we can conclude that $a_{\alpha_i,\beta_i}\not= 0$. Since are we assuming $a_{\alpha_i,\beta_{i+1}} \not= 0$ then $a_{\alpha_i,p}\not=0$.

  In this way we have
      $$
      a_{\alpha_1,\beta_2}\not=0, \dots, a_{\alpha_{i-1}\beta_i}\not=0 , \ a_{\alpha_i,p}\not=0,\ a_{\alpha_{i+1}\beta_{i+1}}\not=0,\dots, \ a_{\alpha_k,\beta_k}\not=0.
      $$
      These elements make up the main diagonal of the submatrix $A\left[\alpha|\tau \cup \{p,\beta_k\}\right]$,  {which suggests that} this submatrix is nontrivial. By the induction hypothesis, we can conclude that
      $$\varepsilon_k \det A\left[\alpha|\tau \cup \{p,\beta_k\}\right] >0.$$

  In addition,
      $$
      \left\{\begin{array}{c}
                 a_{\alpha_1,\beta_2}\not=0, \dots, a_{\alpha_{i-1}\beta_i}\not=0 , \ a_{\alpha_i,p}\not=0,\ a_{\alpha_{i+1}\beta_{i+1}}\not=0,\dots, \ a_{\alpha_{k-1},\beta_{k-1}}\not=0, \\[3mm]
                 a_{\alpha_1,\beta_1}\not=0, \dots, a_{\alpha_{i-1}\beta_{i-1}}\not=0 , \ a_{\alpha_i,\beta_i}\not=0,\ a_{\alpha_{i+1}\beta_{i+1}}\not=0,\dots, \ a_{\alpha_{k-1},\beta_{k-1}}\not=0.
             \end{array}
      \right.
      $$
 Denoting as $w=(\alpha_1,\dots, \alpha_{k-1})$, we have that the submatrices $A\left[ w| \tau \cup\{p\} \right]$ and $A \left[ w | \tau \cup \{\beta_1\}\right]$ are nontrivial submatrices. By the induction hypothesis on $k$, we obtain
 $$
 \varepsilon_{k-1}\det A\left[ w| \tau \cup\{p\} \right] >0,\ \varepsilon_{k-1}\det A \left[ w | \tau \cup \{\beta_1\}\right] >0.
 $$

As for $A\left[w|\tau \cup \{\beta_k\}\right]$, we cannot  {affirm} that it is nontrivial. Considering that $A$ is type-I staircase and that the submatrix is of order $k-1$, we can deduce that
$$
      \varepsilon_{k-1}\det A\left[ w| \tau \cup\{\beta_k\} \right] \geq 0.
$$
By (\ref{igAndo}), we can write
$$
   \hspace*{-3cm}\varepsilon_{k-1} \varepsilon_{k}\det A\left[w|\tau \cup \{p\}\right] \det A\left[\alpha|\tau \cup \{\beta_1,\beta_k\}\right]=
   $$
   $$
      \hspace*{-3cm}=\underbrace{\varepsilon_{k-1}\det A\left[w|\tau \cup \{p\}\right]}_{> 0} \varepsilon_{k}\det A\left[\alpha|\tau \cup \{\beta_1,\beta_k\}\right]=
$$
$$
 \hspace*{-3cm}=\varepsilon_{k-1} \varepsilon_{k}\det A\left[w|\tau \cup \{\beta_k\}\right]\det A\left[\alpha |\tau \cup \{\beta_1,p\}\right]+
$$
$$
\hspace{2cm}+\varepsilon_{k-1} \varepsilon_{k}\det A\left[w|\tau \cup \{\beta_1\}\right]\det A\left[\alpha |\tau \cup \{p,\beta_k\}\right]=
$$
$$
 \hspace*{-3cm}=\underbrace{\varepsilon_{k-1}\det A\left[w|\tau \cup \{\beta_k\}\right]} _{\geq 0} \underbrace{\varepsilon_{k}\det A\left[\alpha |\tau \cup \{\beta_1,p\}\right]}_{>0}+
$$
$$
\hspace{2cm}+\underbrace{\varepsilon_{k-1} \det A\left[w|\tau \cup \{\beta_1\}\right]}_{>0}\underbrace{\varepsilon_{k}\det A\left[\alpha |\tau \cup \{p,\beta_k\}\right]}_{>0}.
$$
Therefore
$$
\varepsilon_{k}\det A\left[\alpha|\tau \cup \{\beta_1,\beta_k\}\right]>0
$$
and the result holds.

\item[Case(2)] Now, let us suppose that $a_{\alpha_s,\beta_{s+1}}=0$ for some $1 \leq i \leq k-1$. Taking into account that $A$ is a type-I staircase matrix and $A\left[\alpha|\beta\right]$ is a nontrivial submatrix, it holds that $a_{\alpha_s,\beta_{s+1}}$ generates a shadow of zeros upwards, therefore
      $$
      A\left[\alpha|\beta\right]=
      \left(
      \begin{array}{cc}
      A\left[\alpha_1,\dots,\alpha_s|\beta_1,\dots, \beta_s\right] & 0 \\
      * & A\left[\alpha_{s+1},\dots,\alpha_k | \beta_{s+1},\dots,\beta_k \right]
      \end{array}
      \right).
      $$
      Thus, the nontrivial minor $\det A\left[ \alpha | \beta \right]$ can be  {written} as
      $$
      \det A\left[ \alpha | \beta \right]=\det A\left[\alpha_1,\dots,\alpha_s |\beta_1,\dots, \beta_s \right] \det A\left[\alpha_{s+1},\dots,\alpha_k | \beta_{s+1},\dots,\beta_k \right],
      $$
      where each of the minors on the right side of the equality are nontrivial, since its diagonal is part of that of $A\left[ \alpha | \beta \right]$. Besides, as $a_{\alpha_s,\beta_{s+1}}=0$ then
      $$
      h=\min \left\{ |i-j| \ / \ a_{ij}=0, \ i \in \{1,2,\dots,m\},\ j\in \{1,2,\dots,n\} \right\}>0
      $$
      and it is possible to apply Lemma \ref{lema9}, and if $\forall 1 \leq k \leq \hbox{rank} (A)-h+1$ we can conclude that $\varepsilon_k=\varepsilon_1^{k}$.
      Note that  $k \leq \alpha_s + n-\beta_{s+1}+1$, so, as the condition (\ref{def_assr_gen}) is satisfied  for $k-1$, we have that
      $$
      \varepsilon_k\det A[\alpha |\beta] =\varepsilon_1^k\det A[\alpha |\beta] =
      $$
      $$=\varepsilon_1 ^s \det A[\alpha_1,\dots , \alpha_s |\beta_1,\dots , \beta_s] \varepsilon_1^{k-s} \det A[\alpha_1,\dots , \alpha_s |\beta_1,\dots , \beta_s]>0.
      $$
Therefore the condition (\ref{def_assr_gen}) has been proved when it is satisfied that $\alpha \in Q_{k,m}$ with $d(\alpha)=0$ (that is, $\alpha \in Q_{k,m}^0$).

Applying a similar reasoning by rows, (\ref{def_assr_gen}) is satisfied in general. Then the matrix $A$ is ASSR with signature $\varepsilon$.
\end{description}
If the matrix $A$ is type-II staircase, it would suffice to apply Theorem \ref{TeoremaMenores} to the matrix $B=P_mA$.
\end{proof}

The following result considers the  {transposition} of ASSR matrices.

\begin{propo}
Let $A$ be an $m \times n$ ($m\geq n$) with $\hbox{rank} (A)=r= \min \{m,n\}$. Then, $A$ is an ASSR matrix with signature $\varepsilon=(\varepsilon_1,\dots, \varepsilon _r)$ if and only if $A^T$  is also an ASSR matrix with the same signature.
\end{propo}

\begin{proof}
Let us suppose that $A$ is a type-I staircase matrix. Note that  {for} any nontrivial submatrix of $A$ ($A[\alpha|\beta]$, with $1 \leq s \leq r$,  $\alpha\in Q_{s,m}$, $\beta \in Q_{s,n}$),  {its diagonal entries are} nonzero, and such entries are the same for the submatrix  $(A[\beta|\alpha])^T$. This submatrix is a submatrix of $A^T$ obtained as $A^T[\beta | \alpha]$. In addition, the nontrivial submatrices of $A^T$ have  associated nontrivial submatrices in $A$. Thus, since $\det A[\alpha|\beta]=\det A^T[\beta | \alpha]$, the two matrices are ASSR and have the same signatures.

In the case that $A$ is a type-II staircase matrix, it is sufficient to note that $P_mA$ is type-I staircase to obtain the result.
\end{proof}

\section{$QR$ factorization of ASSR matrices}

In this section we are going to study the $QR$ factorization of a rectangular matrix with linearly independent columns (see \cite{Golub2013}).

\begin{defi}\label{facQR}
Let $A=\left(a_{ij}\right)_{1 \leq i \leq m, 1 \leq j \leq n}$ be an $m\times n$ real matrix with $m \geq n$ and $\hbox{rank}(A)=n$. A $QR$ factorization of $A$ is a decomposition of the form $A=QR$, where $Q$ is an $m\times n$ real matrix verifying $Q^TQ=I_n$ and $R$ is an $n \times n$ upper triangular matrix.
\end{defi}

This type of factorization for matrices with linearly independent columns is always possible, as can be seen in the following result that appears in various Linear Algebra texts (see, for instance \cite{Strang2006})

\begin{thm}\label{existeQR}
Let $A$ be an $m\times n$ matrix with linearly independent columns. Then $A$ can be factored as $A=QR$, where $Q$ is an $m\times n$ matrix with orthonormal columns and $R$ is an $n \times n$ invertible upper triangular matrix with positive diagonal entries. If $m=n$ then $Q$ and $R$ are square matrices and $Q$ is an orthogonal matrix.
\end{thm}

The following result gives us important information about the upper triangular matrix $R$ of the factorization $QR$ of a SR matrix with maximum rank.

\begin{thm}\label{RTP}
Let $A$ be an $m\times n$ real matrix with $m \geq n$, $\hbox{rank}(A)=n$ and SR with signature  $\varepsilon$. Then $A$ can be factored as $A=QR$ where $Q$ is an $m\times n$ matrix verifying that $Q^TQ=I_n$ and $R$ an $n \times n$ nonsingular upper triangular TP matrix.
\end{thm}

\begin{proof}
Since $A$ is such that $\hbox{rank}(A)=n$, Theorem \ref{existeQR} allows us to ensure that the factorization $A=QR$ exists in the terms established in Definition \ref{facQR}.

In addition, $A^TA=R^TQ^TQR=R^TR$, and since $A$ is an SR matrix with signature $\varepsilon$  {and so is its transpose ($A^T$) and, consequently,} with the same signature, and by Theorem 3.1  {of} \cite{Ando1987} we can assure that $A^T A$ is also SR with signature $\varepsilon \varepsilon=\varepsilon^2=(1,1,\dots, 1)$. Thus, $R^TR$ is an SR matrix with signature $\varepsilon=(1,1,\dots, 1)$ and therefore TP.

On the other hand, since $R=(r_{ij})_{1 \leq i,j \leq n}$ is an upper triangular matrix with positive diagonal elements ($r_{ii} >0$, for all $i$), this matrix can be factored as $R=DU$, with $D=(d_{ij})_{1 \leq i,j \leq n}$ a diagonal matrix with positive diagonal elements, and $U=(u_{ij})_{1 \leq i,j \leq n}$ an upper triangular with $u_{ii}=1$ for $1 \leq i \leq n$, and thus $R^TR=U^TD^TDU=L \bar{D} U$.

Now, considering that $R^TR$ is TP, it is possible to apply the Theorem 5.4 of \cite{Gasca1996TP}, so that it can be decomposed in the form $L \bar{D} U$ with $L=U^T$ and $\bar D=D^TD$, being this the unique factorization $LDU$ of the nonsingular matrix $R^T R$. Besides, $\bar{D}$ is a diagonal matrix with positive diagonal elements, $L$ is a lower triangular TP with unit diagonal elements and $U$ is an upper triangular TP with $u_{ii}=1$for $1 \leq i \leq n$. Thus, we can conclude that $R=DU$ is TP.
\end{proof}

\begin{ejem}
Let $A$ be the  {following} $4\times 3$ matrix with $\hbox{rank}(A)=3$ and SR with signature $\varepsilon=(-1,1,-1)$
$$
A=\left(
    \begin{array}{ccc}
    -3 & -1 & 0  \\
    -2 & -5 & -2  \\
    -10^{-6} -1 & -1 & -1  \\
    0  & 0 & -1
    \end{array}
  \right).
$$
Using the following function of \texttt{MatLAb}

\begin{verbatim}
function [Q,R] =  mgs(A)
    % Modified Gram-Schmidt.  [Q,R] = mgs(A);
    % G.W. Stewart, "Matrix Algorithms, Volume 1", SIAM, 1998.
    [m,n] = size(A);
    Q = zeros(m,n);
    R = zeros(n,n);
    for k = 1:n
        Q(:,k) = A(:,k);
        for i = 1:k-1
            R(i,k) = Q(:,i)'*Q(:,k);
            Q(:,k) = Q(:,k) - R(i,k)*Q(:,i);
        end
        R(k,k) = norm(Q(:,k))';
        Q(:,k) = Q(:,k)/R(k,k);
    end
end
\end{verbatim}
the following results have been obtained
$$
Q=\left(
    \begin{array}{ccc}
-0.832050294337812  &  0.534522583680626  & -0.068278727592371 \\
-0.554700196225208  & -0.801783741890443  &  0.102418313294977 \\
-0.000000277350098  & -0.267260993741225  & -0.443812838882511 \\
                  0 &                   0 & -0.887625297354253
  \end{array}
\right),
$$
$$
R=\left(
    \begin{array}{ccc}
3.605551275464128 &  3.605551552813948 &  1.109400669800514 \\
 0  & 3.741657119512813 &  1.870828477522111 \\
 0  &      0  & 1.126601509646810
  \end{array}
\right),
$$
which allows us to verify that $Q$ is a matrix $4\times 3$ that fulfills $Q^TQ=I_3$ and $R$ is a nonsingular upper triangular matrix $3 \times 3$ TP.
\end{ejem}

 {
\begin{rmk} Note, that modified Gram-Schmidt algorithm (mgs) is not just a simple modification of classical Gram-Schmidt algorithm (see \cite{Stewart1998}). The idea is to orthogonalize against the emerging set of vectors instead of against the original set. There are two variants, a column-oriented one and a row-oriented one. They produce the same results, in different order. In this case, we have considered the column version.
\end{rmk}
}

Let us now see that the rows and columns involved in the column boundary minors of a type-I staircase matrix $A$ correspond to column generalized boundary minors in $Q$.

\begin{propo}\label{lematipo-I}
Let $A=\left(a_{ij}\right)_{1 \leq i \leq m, 1 \leq j \leq n}$ be  {an} $m\times n$ real matrix with $m \geq n$ and $\hbox{rank}(A)=n$. Let $A=QR$ be a $QR$ factorization of $A$, such that the main diagonal of $R$ is strictly positive. Then, for $k\leq n$, if $\det A[\alpha | \beta]$ with $\alpha \in Q^0_{k,m}$ and $\beta \in Q^0_{k,n}$ is a column boundary minor of $A$, then $\det Q[\alpha |\beta]$ is a column generalized boundary minor of $Q$.
\end{propo}

\begin{proof}
Since $A$ verifies the conditions of Theorem \ref{existeQR}, we can assure that there is a factorization $QR$. Besides, if we perform the factorization using the Gram-Schmidt method, it is ensured that the entries of the main diagonal of $R$ are strictly positive.  {By} applying this method, for each $A[1,2,\dots,m|j]$, $j$-th  column of the matrix $A$, the unique values $x_1,x_2,\dots , x_{j-1}\in \mathbb{R}$ are calculated with
\begin{eqnarray}\label{defHj}
H[1,\dots,m|j]& = & A[1,\dots,m|j]-x_1A[1,\dots,m|1]-x_2A[1,\dots,m|2]-\cdots   \nonumber \\
 &  & - x_{j-1}A[1,\dots,m|j-1]
\end{eqnarray}
such that $H[1,2,\dots,m|j]^T A[1,2,\dots,m|k]=0$, $\forall k\in \{1,\cdots , j-1\}$. As is known, if $A$ has rank $n$, the $n$ columns of $A$ form a free system, and since the column $j$ of $H$ is a linear combination of the first $j$ columns of $A$ with the coefficient of the column $j$  {being} nonzero, then $H[1,\dots,m|j]\not=0$, and as a result $\| H[1,\dots,m|j] \|\not=0$. Thus,  normalizing $H[1,\dots,m|j]$, we construct $Q[1,\dots,m|j]=\dfrac{H[1,\dots,m|j]}{\| H[1,\dots,m|j] \|}$, so the zero-nonzero entries of column $j$ of $H$ match those in the $j$ column of $Q$.

Since $A$ is a type-I staircase matrix, if $a_{ij}=0$ with $i > j$, then $a_{i,j-1}=0$, $a_{i,j-2}=0, \dots , a_{i,1}=0$.  {For} this reason and by (\ref{defHj}) we have that
\begin{equation}\label{qij=0}
h_{ij}=a_{ij}-x_1a_{i1}-x_2a_{i2}-\cdots - x_{j-1}a_{i,j-1}=0 \Rightarrow q_{ij}=0.
\end{equation}

Suppose now that $a_{i,j-1}=0$, $a_{i,j-2}=0, \dots , a_{i1}=0$ and $a_{ij}\not=0$, then
\begin{equation}\label{qijnot=0}
h_{ij}=a_{ij}-x_1a_{i1}-x_2a_{i2}-\cdots - x_{j-1}a_{i,j-1}=a_{ij} \Rightarrow q_{ij}\not=0.
\end{equation}
So, it is evident that if $\det A[\alpha | \beta]$ is a column boundary minor of $A$ we have either $\beta_1=1$, and so $\det Q[\alpha|\beta]$ is a column generalized boundary minor of $Q$, or, since the matrix is type-I staircase, $A[\alpha|\beta_1-1]=0$, which would imply $ Q[\alpha|\beta_1-1]=0$ and, at least $a_{\alpha_1,\beta_1}\not=0$, so that $q_{\alpha_1,\beta_1}\not=0$.
\end{proof}

\begin{ejem}\label{ejQR}
Let $A$ be the following matrix  {which allows for} a $QR$ factorization
  $$
  A=\left(\begin{array}{ccccc} -1 & -4 & 0 & 0 & 0\\ -2 & -10 & -10 & -16 & -2\\ 0 & -6 & -33 & -60 & -21\\ 0 & -8 & -46 & -92 & -70\\ 0 & 0 & -9 & -60 & -242\\ 0 & 0 & -6 & -60 & -443 \end{array}\right)=QR
  $$
with
  $$
  Q=\left(
      \begin{array}{rrrrr}
        -0.4472  &    0.0797 &     -0.0481 &     0.0987 &    -0.7699 \\
        -0.8944  &    -0.0398 &     0.0240 &    -0.0493 &     0.3849 \\
              0  &    -0.5976 &     0.0852 &    -0.0887 &     0.3285 \\
              0  &    -0.7968 &    -0.0699 &     0.0789 &    -0.3426 \\
              0  &          0 &    -0.8258 &     0.5313 &     0.1684 \\
              0  &          0 &    -0.5505 &    -0.8315 &    -0.0742 \\
      \end{array}
    \right),
  $$
  $$
  R=\left(
      \begin{array}{ccccc}
      2.2361       &    10.7331 &     8.9443 &    14.3108 &     1.7889 \\
              0    &    10.0399 &    56.7734 &   109.8017 &    68.4069 \\
              0    &          0 &    10.8989 &    83.5122 &   446.7713 \\
              0    &          0 &          0 &    16.8671 &   236.2195 \\
              0    &          0 &          0 &          0 &     8.4289 \\
      \end{array}
    \right).
  $$
Note that the column boundary minor of $A$  {corresponds} to the column generalized boundary minor of $Q$. The same does not occur with the row boundary minor, since, for example, $\det A[2,3|3,4]$ is a row boundary minor of $A$, because $A[1|3,4]=[0, 0]$, but $Q[1|3,4]=[-0.0481, 0.0987]\not= [0,0]$.
\end{ejem}

The following result analyzes the $QR$ factorization of ASSR matrices.

\begin{thm}
Let $A$ be an $m\times n$ real matrix with $m \geq n$ and $\hbox{rank}(A)=n$. Let $A$ be a type-I staircase,  ASSR with signature $\varepsilon =(\varepsilon_1,\varepsilon_2 ,\dots, \varepsilon_n)$, and $A=QR$ a $QR$ factorization of $A$ with $R$ an upper triangular matrix with  {its} main diagonal strictly positive. Then, $R$ is a TP matrix and, furthermore, if $\det A[\alpha |\beta]$ is a column boundary minor of $A$, then  $\det Q[\alpha |\beta]$ is a column generalized boundary minor of $Q$ and $\varepsilon _k \det Q[\alpha |\beta]>0$.
\end{thm}

\begin{proof}
First, it should be noted that if $A$ is an ASSR matrix, then it is SR and therefore it satisfies the conditions of Theorem \ref{RTP}. So, we can affirm that $R$ is a nonsingular TP matrix.

On the other hand, if $\det A[\alpha | \beta]$ is a column boundary minor of $A$, the elements of its diagonal are nonzero (by Definition \ref{sub-frontera}), and therefore the minor is nontrivial. In addition, as $A$  {is} ASSR, the minor is nonzero.

By the Cauchy-Binet formula (cf. equation (1.23) of \cite{Ando1987}) we have
\begin{equation}\label{c-b}
\det A[\alpha | \beta]=\sum_{\omega \in Q_{k,n}} \det Q[\alpha | \omega ] \det R[\omega | \beta].
\end{equation}

If $\beta_1=1$ then $\beta =(1,2,\dots,k)$, and since $R$ is an upper triangular matrix, $\det R[\omega | \beta]=0$ if $\omega \not= \beta$ so that
$$
\underbrace{\det A[\alpha | \beta]}_{\not= 0}= \det Q[\alpha | \beta ] \underbrace{\det R[\beta | \beta]}_{\not= 0} \Rightarrow \det Q[\alpha | \beta ]\not= 0
$$
Let us suppose now that $\beta_1 >1$. Then, by (\ref{c-b}) and considering that $R$ is an upper triangular matrix we have that $\det R[\omega | \beta]=0$ if any $\omega_i > \beta_i$.

Suppose that $\omega_1 < \beta_1 \leq \alpha_1$. Since $\det A[\alpha | \beta]$ is a column boundary minor of $A$ with $\beta_1>1$ then it holds that $ A[\alpha|\beta_1-1]=0$.
By Proposition \ref{lematipo-I} $\det Q[\alpha |\beta]$ is a column generalized boundary minor of $Q$, and so $Q[\alpha|\beta_1-1]=0$. Then, given that $a_{\alpha_1,\beta_1-1}=0$ and $A$ is type-I staircase, we have that $a_{\alpha_i,\omega_1}=0$, for all $i$, hence $q_{\alpha_i,\omega_1}=0$  and $\det Q[\alpha | \omega ]=0$, therefore the only nonzero term of (\ref{c-b}) corresponds to $\omega$ with $\omega_1=\beta_1$.

Next, we suppose that $\omega_2<\beta_2 \leq \alpha_2$. By Definition \ref{sub-frontera} it is fulfilled that $\alpha\in Q^{0}_{k,m}$ and $\beta \in Q^{0}_{k,n}$, from which $\alpha_i=\alpha_1+i-1$ y $\beta_i=\beta_1+i-1$ and thus $\omega_2<\beta_2=\beta_1+1\leq\alpha_2 = \alpha_1+1$. So, $\omega_2-1 < \beta_1-1$, and since $A[\alpha|\beta_1-1]=0$, we have $a_{\alpha_2,\omega_2-1}=0$. In addition, since $A$ is a type-I staircase matrix, we obtain that $a_{i,\omega_2-1}=0$ by each $i \geq \alpha_2$, and by (\ref{qij=0}) it is fulfilled that  $q_{i,\omega_2-1}=0$, so, $\det Q[\alpha_2,\dots, \alpha_k|\omega_2,\dots,\omega_k]=0$. Then
$$
\det Q[\alpha|\omega]=q_{\alpha_1,\omega_1}\det Q[\alpha_2,\dots, \alpha_k|\omega_2,\dots,\omega_k]=0.
$$
Hence the only nonzero term of (\ref{c-b}) is that associated with $\omega=(\beta_1,\beta_2,\omega_3,\dots \omega_k)$.

 {By} reasoning in the same way with the remaining elements of $\omega$ we have that the only nonzero term in the Cauchy-Binet formula is the one that corresponds to $\omega=\beta$, so
$$
\underbrace{\det A[\alpha | \beta]}_{\not= 0}= \det Q[\alpha | \beta ] \underbrace{\det R[\beta | \beta]}_{\not= 0} \Rightarrow \det Q[\alpha | \beta ]\not= 0.
$$
Besides, since $A$ is an ASSR matrix with signature $\varepsilon =(\varepsilon_1,\varepsilon_2 ,\dots, \varepsilon_n)$, by Definition \ref{def_assr_gen} we have that
$$
0<\varepsilon_k \det A[\alpha | \beta]=\varepsilon_k \det Q[\alpha | \beta ] \underbrace{\det R[\alpha | \alpha]}_{>0}.
$$
So, if $\det A[\alpha |\beta]$ is a column boundary minor of $A$, then $\det Q[\alpha |\beta]$ is a generalized column boundary minor of $Q$ of order $k$, and $\varepsilon _k \det Q[\alpha |\beta]>0$.
\end{proof}

\begin{ejem}
The condition given in the previous result on $Q$ does not ensure the sign regularity of $Q$.  {Considering} the matrices of Example \ref{ejQR}  {again}, we see that the elements of $Q$  {do not} have the same sign, and so the minors of order $1$ do not verify the condition.
\end{ejem}

\section*{Acknowledgements}
This work has been partially supported by the Spanish Research Grant PGC2018-096321-B-100, TIN2017-87600-P and MTM2017-90682-REDT.


\end{document}